	\documentclass[12pt,a4paper]{amsart}
	\usepackage{graphicx, latexsym, amsmath, amsthm, amssymb}

       \def\pvs{\par\vspace*}
   \def\nid{\noindent}
\def\pvc{\pvs{ 3mm}} \def\pvcn{\pvc \nid} 

\theoremstyle{plain}
\newtheorem{thm}{Theorem}[section]
\newtheorem{lem}[thm]{Lemma}
\newtheorem{prop}[thm]{Proposition}
\newtheorem{cor}[thm]{Corollary}
\theoremstyle{definition}

\theoremstyle{remark}
\newtheorem{rem}[thm]{Remark}

\newtheorem{example}[thm]{Example}
\newtheorem*{acknowledgments}{Acknowledgments}
\numberwithin{equation}{section}
\numberwithin{figure}{section}
\newcommand{\ov}{\overline}

	  \title[Homotopy, $\Delta$-equivalence and concordance for knots]
		{Homotopy, $\Delta$-equivalence and concordance for knots 
		in the complement of a trivial link}
	 \author[Thomas FLEMING {et al.}]
		{Thomas FLEMING, Tetsuo SHIBUYA, Tatsuya Tsukamoto \\
		 and Akira YASUHARA}
	   \date{}
	 \thanks{The last author is partially supported by a Grant-in-Aid 
	 	 for Scientific Research (C) ($\#$20540065) of the Japan 
	 	 Society for the Promotion of Science.}
	 	 
\begin{document}
			\maketitle
			
\begin{center} {\it Dedicated to Professor Kunio Murasugi on his 80th birthday}
\end{center}

\begin{abstract}
Link-homotopy and self $\Delta$-equivalence are equivalence
relations on links. It was shown by J. Milnor (resp. the last
author) that Milnor invariants determine whether or not a link is
link-homotopic (resp. self $\Delta$-equivalent) to a trivial link.
We study link-homotopy and self $\Delta$-equivalence on a certain
component of a link with fixing the rest components, in other words,
homotopy and $\Delta$-equivalence of knots in the complement of a
certain link. We show that Milnor invariants determine whether a
knot in the complement of a trivial link is null-homotopic, and
give a sufficient condition for such a knot to be
$\Delta$-equivalent to the trivial knot. We also give a sufficient
condition for knots in the complements of the trivial knot to be
equivalent up to $\Delta$-equivalence and concordance.
\end{abstract}

\section{Introduction}


For an ordered and oriented $n$-component link $L$, the {\em Milnor invariant}
$\overline{\mu}_L(I)$ is defined for each multi-index $I=i_1i_2...i_m$ with
entries from $\{1,...,n\}$ \cite{Milnor,Milnor2}. Here $m$ is called the 
{\em length} of $\overline{\mu}_L(I)$ and denoted by $|I|$. Let $r(I)$ denote 
the maximum number of times that any index appears in $I$. Hence any index 
appear in $I$ at most $r(I)$ times. It is known that if $r(I)=1$, then 
$\overline{\mu}_L(I)$ is a {\em link-homotopy} invariant \cite{Milnor}, where 
{\em link-homotopy} is an equivalence relation on links generated by self 
crossing changes.

While Milnor invariants are not strong enough to give a link-homotopy 
classification for links, they determine whether a link is link-homotopic 
to a trivial link or not. In fact, it is known that a link $L$ in $S^3$ is 
link-homotopic to a trivial link if and only if $\overline{\mu}_L(I)=0$ for 
any $I$ with $r(I)=1$ \cite{Milnor,HL}.


Even if a link is link-homotopic to a trivial link, it is not necessarily 
true that a certain component of the link is null-homotopic in the complement 
of the other components. In this paper, we study homotopy of knots in the 
complement of a certain link.

Although Milnor invariants $\overline{\mu}(I)$ with $r(I)\geq 2$ are not 
necessarily link-homotopy invariants, we have the following. The \lq only if' 
part holds for more general setting, see Proposition~\ref{free-Ck-inv}.

\begin{thm}\label{free-homotopy}
Let $L=K_0\cup K_1\cup\cdots\cup K_n$ be an $(n+1)$-component link 
such that $L-K_0$ is a trivial link. 
Then $K_0$ is null-homotopic 
in $S^3\setminus(L-K_0)$ if and only if $\ov{\mu}_L(I0)=0$ 
for any multi-index $I$ with entries from $\{1,...,n\}$.
\end{thm}

\begin{rem}\label{remark1}
(1)~In the theorem above the condition that $L-K_0$ is a trivial link is
essential.
Let $K$ be a non-trivial knot and $K'$ be the longitude of
a tubular neighbourhood of $K$.
Then the link $L=K\cup K'$ is a {\em boundary link}, i.e.,
its components bound disjoint orientable surfaces.
Hence the all Milnor invariants of $L$ vanish.
(Note that $L$ is link-homotopic to a trivial link.)
On the other hand, since $K$ is a non-trivial knot,
it follows from Dehn's lemma that
$K'$ is not null-homotopic in $S^3\setminus K$
\cite[Chapter~4, B.2]{R}.\\
(2)~In \cite[Example~6.4]{Yasu2}, the last author gave a 3-component link
$L=K_1\cup K_2\cup K_3$ such that $K_i$ is null-homotopic in
$S^3\setminus(L-K_i)~(i=2,3)$ and $K_1$ is not null-homotopic in
$S^3\setminus(L-K_1)$.
\end{rem}

A link is  {\em Brunnian} if every proper sublink of it is trivial.
In particular, trivial links are Brunnian. By
Theorem~\ref{free-homotopy}, we have the following corollary. This
gives a characterization of Brunnian links, where each component is
null-homotopic in the complement of the rest of the components.

\begin{cor}\label{brunnian-free-homotopy}
For an  $n$-component Brunnian link $L$,
the $i$th component $K$ is null-homotopic in $S^3\setminus(L-K)$
if and only if $\overline{\mu}_L(Ii)=0$ for any multi-index $I$
with entries from $\{1,...,n\}\setminus\{i\}$.
\end{cor}

\begin{rem}\label{remark2}
In the last section, we give a 3-component Brunnian link $L$ such
that $L$ is link-homotopic to a trivial link, and each component $K$
of $L$ is not null-homotopic in $S^3\setminus(L-K)$
(Example~\ref{example1}). There are no such examples for 2-component
links, since a knot in the complement of the trivial knot is
null-homotopic if and only if it is null-homologous. Hence, for a
2-component Brunnian link, the following three conditions are
mutually equivalent: (i)~It is link-homotopic to a trivial link.~
(ii)~The linking number vanishes. (iii)~each component is
null-homotopic in the complement of the other component.
\end{rem}

Let $L=K_0\cup K_1\cup\cdots\cup K_n$ be an $(n+1)$-component link.
If $L-K_0$ bounds a disjoint union $F$ of orientable surfaces 
$F_1,..., F_n$ with $\partial F_i=K_i~(i=1,...,n)$ and 
$F\cap K_0=\emptyset$, then by \cite[Section 6]{C}, $\overline{\mu}_{L}(I0)=0$
for any multi-index $I$ with entries from $\{1,...,n\}$.
By combining this and Theorem~\ref{free-homotopy},
we have the following corollary.

\begin{cor}\label{free-homotopy2}
Let $L=K_0\cup K_1\cup\cdots\cup K_n$ be an $(n+1)$-component link 
such that $L-K_0$ is a trivial link. If $L-K_0$ bounds a disjoint union 
$F$ of orientable surfaces $F_1,..., F_n$ with $\partial F_i=K_i~(i=1,...,n)$
and $F\cap K_0=\emptyset$, then $K_0$ is null-homotopic in $S^3\setminus(L-K_0)$.
\end{cor}

\begin{rem}\label{remark3}
J. Hillman has pointed out that Corollary~\ref{free-homotopy2} can be shown 
by using the universal covering space of $S^3\setminus(L-K_0)$ as follows: 
We may construct the maximal free cover of $S^3\setminus(L-K_0)$ by gluing 
infinite copies of $S^3$-cut-along-$F$, for example see \cite[Section 2.2]{Hill}.
Note that the maximal free cover is the universal cover, since the link
$\partial F=L-K_0$ is trivial.
If $K_0\cap F=\emptyset$, then $K_0$ lifts to the universal cover,
and hence is null-homotopic in $S^3\setminus(L-K_0)$.
\end{rem}

Two $n$-component links $L_0$ and $L_1$ are {\em concordant} if there 
are mutually disjoint $n$ annuli $A_1,..., A_n$ in $S^3\times[0,1]$
with $(\partial (S^3\times[0,1]), \partial A_j)=
(S^3\times\{0\},K_{0j})\cup(-S^3\times\{1\},-K_{1j})$ $(j=1,...,n)$,
where $-X$ denotes $X$ with the opposite orientation.
A link is {\em slice} if it is concordant to a trivial link.
Since the Milnor invariants are concordance invariants \cite{Casson},
Theorem~\ref{free-homotopy} gives us the following corollary.

\begin{cor}\label{brunnian-slice}
For any Brunnian, slice link $L$, each component $K$ is null-homotopic
in $S^3\setminus(L-K)$.
\end{cor}

\begin{rem}
Let $K$ be a slice knot which is non-trivial, and $K'$ the longitude
of a tubular neighbourhood of $K$. Then the 2-component link
$L=K\cup K'$ is a slice link. As we saw in Remark~\ref{remark1}~(1),
each component is not null-homotopic in the complement of the other.
Hence the Brunnian property in Corollary~\ref{brunnian-slice} is
necessary.
\end{rem}

A {\em $\Delta$-move} \cite{MN,Mat} is a local move on links as
illustrated in Figure~\ref{delta-move}. If the three strands in
Figure~\ref{delta-move} belong to the same component of a link, 
we call it a {\em self $\Delta$-move} \cite{Shi}. Two links are said 
to be {\em $\Delta$-equivalent} (resp. {\em self $\Delta$-equivalent}) 
if one can be transformed into the other by a finite sequence of
$\Delta$-moves (resp. self $\Delta$-moves). Note that self
$\Delta$-equivalence implies link-homotopy, i.e., if two links are
self $\Delta$-equivalent, then they are link-homotopic. For knots,
self $\Delta$-equivalence is the same as $\Delta$-equivalence.

It is known that a link $L$ in $S^3$ is self $\Delta$-equivalent to a 
trivial link if and only if $\overline{\mu}_L(I)=0$ for any $I$ with 
$r(I)\leq 2$ \cite[Corollary~1.5]{Yasu2}. Even if a link is self 
$\Delta$-equivalent to a trivial link, it is not necessarily true that
a certain component of the link is $\Delta$-equivalent to the trivial 
knot in the complement of the rest components, where a knot is 
{\em trivial} in the complement of a link if it bounds a disk disjoint 
from the link. We study $\Delta$-equivalence of knots in the complement 
of a certain link.
\pvcn
\begin{figure}[!h]
\includegraphics[trim=0mm 0mm 0mm 0mm, width=.45\linewidth]{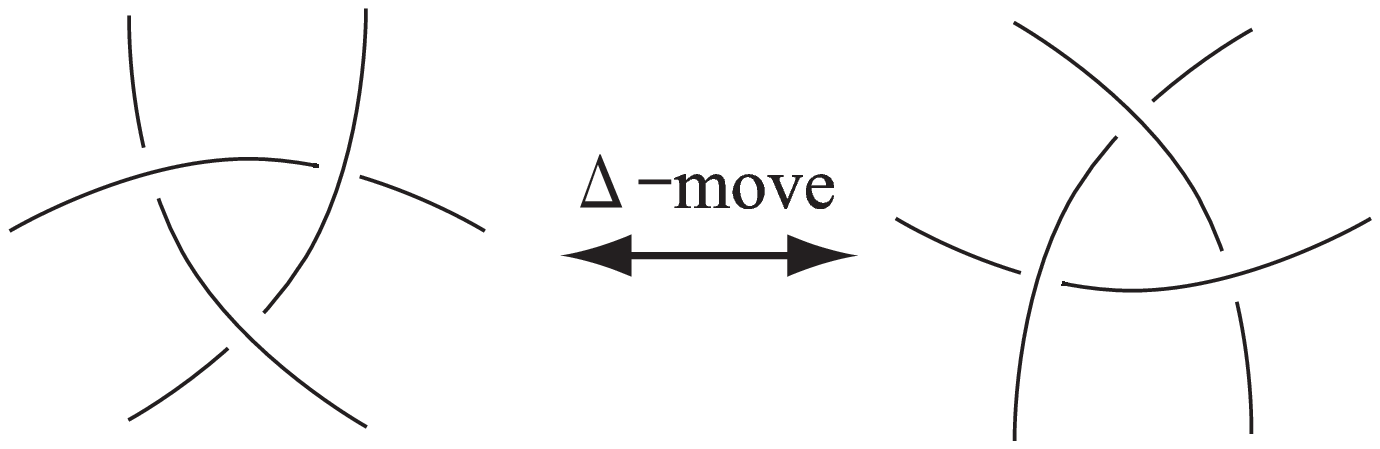}
\caption{}\label{delta-move}
\end{figure}



The following theorem is comparable to Corollary~\ref{free-homotopy2}.

\begin{thm}\label{free-delta}
Let $L=K_0\cup K_1\cup\cdots\cup K_n$ be an $(n+1)$-component boundary 
link such that $L-K_0$ is a trivial link. Then $K_0$ is $\Delta$-equivalent 
to the trivial knot in $S^3\setminus(L-K_0)$. In particular, for any 
Brunnian, boundary link, each component is $\Delta$-equivalent to
the trivial knot in the complement of the rest components.
\end{thm}

\begin{rem}
(1)~As we saw in Remark~\ref{remark1}~(1), there is a 2-component boundary 
link such that each component is not null-homotopic in the complement of
the other component. Since self $\Delta$-equivalence implies link-homotopy,
any component is not $\Delta$-equivalent to the trivial knot in the complement
of the other component. This implies that the condition, $L-K_0$ is trivial, 
in Theorem~\ref{free-delta} is essential. We also notice by \cite{SY, Yasu2} 
that $L$ is self $\Delta$-equivalent to a trivial link since $L$ is a boundary 
link.  \\
(2)~In the last section, we give a 3-component Brunnian link $L$ such that
$L$ is self $\Delta$-equivalent to a trivial link, and each component $K$ of 
$L$ is not  $\Delta$-equivalent to the trivial knot in $S^3\setminus(L-K)$ 
(Example~\ref{example2}). Since some Milnor invariants of $L$ are non-trivial, 
$L$ is not a boundary link. Hence the condition that $L$ is a boundary link 
in Theorem~\ref{free-delta} is necessary.
\end{rem}


For an $n$-component link $L=K_1\cup\cdots\cup K_n$, we denote by $W^i(L)$
the link with the $i$th component Whitehead doubled. In particular $W^i(K_i)$ 
is the $i$th component of $W^i(L)$. Note that $L-K_i=W^i(L)-W^i(K_i)$.
Then we have the following relation between homotopy of a knot and
$\Delta$-equivalence of the Whitehead double of that knot in the 
complement of a trivial link.

\begin{thm}\label{homotopy-delta} $($cf. \cite[Theorem~1.4]{MY}$)$
Let $L=K_0\cup K_1\cup\cdots\cup K_n$ be an $(n+1)$-component link 
such that $L-K_0$ is a trivial link. The component $K_0$ is null-homotopic 
in $S^3\setminus(L-K_0)$ if and only if $W^0(K_0)$ is $\Delta$-equivalent 
to the trivial knot in $S^3\setminus(L-K_0)$.
\end{thm}

It is known that concordance implies link-homotopy \cite{Gif,Gol}
and it does not necessarily imply self $\Delta$-equivalence
\cite[Claim~4.5]{NS-JKTR00}. Now we consider an equivalence relation
on links combining self $\Delta$-equivalence and concordance. Two
links $L$ and $L'$ are {\em self-$\Delta$ concordant} if there is a
sequence $L=L_1,...,L_m=L'$ of links such that $L_i$ and $L_{i+1}$
are either concordant or self $\Delta$-equivalent for each
$i\in\{1,...,m-1\}$. Links up to self $\Delta$-equivalence and
concordance have been studied in 
\cite{TS-MOIT07a}, and \cite{Yasu}. Classification of {\em string links} 
up to self-$\Delta$ concordance is given by the last author \cite{Yasu}. 
In \cite{TS-KBJM07} and \cite{TS-MOIT07a}, the second  author defined an 
equivalence relation, {\em $\Delta$-cobordism}. It is not hard to see 
that two links are $\Delta$-cobordant if and only if they are
self-$\Delta$ concordant.

We consider self-$\Delta$ concordance of a certain component of a link 
while fixing the rest of the components. i.e., self-$\Delta$ concordance 
of knots in the complement of a certain link. Two knots $K$ and $K'$ in 
the complements of a link $L$ are {\em self-$\Delta$ concordant $($ or 
$\Delta$ concordant $)$ in $S^3\setminus L$} if there is a sequence 
$K=K_1,...,K_m=K'$ of knots such that $K_i$ and $K_{i+1}$ are either 
$\Delta$-equivalent or concordant in $S^3\setminus L$ for each 
$i\in\{1,...,m-1\}$, where $K_i$ and $K_{i+1}$ are concordant in 
$S^3\setminus L$ if there is an annulus $A$ in $(S^3\setminus L)\times[0,1]$
with $(\partial((S^3\setminus L)\times[0,1]),\partial A)=((S^3\setminus L)
\times\{0\}, K_i)\cup(-(S^3\setminus L)\times\{1\},-K_{i+1})$. For knots 
in the complement of the trivial knot in $S^3$, we have the following.

\begin{thm}\label{delta-concordance}
Let $K$ and $K'$ be knots in the complement of the trivial knot $O$ in $S^3$.
If $\mathrm{lk}(K,O)=\mathrm{lk}(K',O)=\pm 1$, then $K$ and $K'$ are $\Delta$ 
concordant in $S^3\setminus O$.
\end{thm}

\begin{rem}\label{remark6}
(1)~Let $K\cup O$ be the link illustrated in Figure~\ref{example2TSY}, where 
$O$ is the trivial knot and $K$ is a trefoil. Let $H=O'\cup O$ be the Hopf 
link with linking number one. Note that $\mathrm{lk}(K,O)=\mathrm{lk}(O',O)=1$.
It follows from \cite[Proposition~2]{N-O} that $K\cup O$ is not self 
$\Delta$-equivalent to $H$. While $K$ is neither $\Delta$-equivalent nor 
concordant to $O'$ in $S^3\setminus O$, the theorem above implies that they 
are $\Delta$ concordant in $S^3\setminus O$. \\
(2)~Let $W=K\cup O$ be the Whitehead link. Then $\overline{\mu}_W(1122)\neq 0$. 
Since $\overline{\mu}(1122)$ is invariant under both self $\Delta$-equivalence 
\cite{FY} and concordance \cite{Casson}, $K$ is not $\Delta$ concordant to be
trivial in $S^3\setminus O$. This implies Theorem~\ref{delta-concordance} does 
not hold for $\mathrm{lk}(K,O)=\mathrm{lk}(K',O)=0$. Moreover, in
Example~\ref{example4}, we show that for any $p~(|p|\geq 2)$, there are two 
links $K\cup O$ and $K'\cup O$ with $\mathrm{lk}(K,O)=\mathrm{lk}(K',O)=p$ 
such that $K\cup O$ and $K'\cup O$ are not self-$\Delta$ concordant. 
In particular, $K$ and $K'$ are not $\Delta$ concordant in $S^3\setminus O$. 
Hence the condition $\mathrm{lk}(K,O)=\mathrm{lk}(K',O)=\pm 1$ is essential.
\end{rem}

\begin{figure}[!h]
\includegraphics[trim=0mm 0mm 0mm 0mm, width=.25\linewidth]{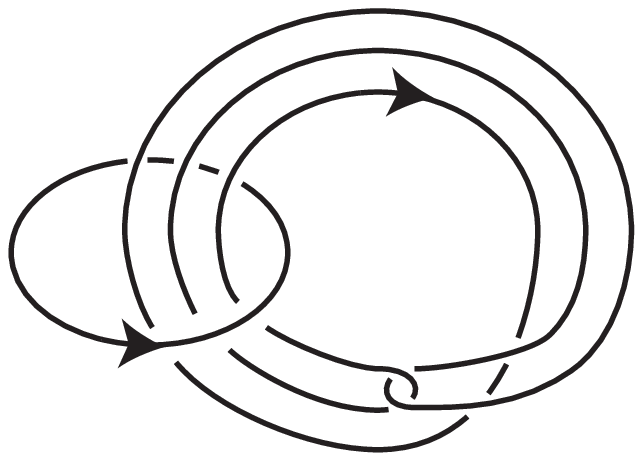}
        \caption{}\label
{example2TSY}
\end{figure}

Let $V_1 \cup\cdots\cup V_n$ be a regular neighborhood of a link
$\Gamma=\gamma_1 \cup\cdots\cup \gamma_n$ in $S^3$. Let ${k}_i$ be a knot 
in an unknotted solid torus $\widetilde{V}_i \subset S^3$ such that ${k}_i$ 
is not contained in a $3$-ball in $\widetilde{V}_i$ $(i=1, \cdots, n)$. 
Let $l_i$ be the linking number of ${k}_i$ and a meridian of $\widetilde{V}_i$.
Let $\phi_i:\widetilde{V}_i \to V_i$ be a homeomorphism which maps a preferred
longitude of $\widetilde{V}_i$ onto a preferred longitude of $V_i$. We call
the image $L=K_1 \cup\cdots\cup K_n=\phi_1({k}_1) \cup\cdots\cup \phi_n({k}_n)$ 
a {\it componentwise satellite link of type  $(\Gamma;l_1,...,l_n)$} and 
$\Gamma$ the {\it companion of} $L$. The link in Figure~\ref{example2TSY} is
a componentwise satellite link of type  $(H;1,1)$ for the Hopf link $H$ with 
linking number one. 
If $l_1=\cdots=l_n=1$, then by Theorem~\ref{delta-concordance}, each $k_i$ 
is $\Delta$ concordant to the core of $\widetilde{V}_i$ in $\widetilde{V}_i$. 
Hence we have the following.

\begin{cor}\label{thm:core}
Let $L$  be a componentwise satellite link of type  $(\Gamma;1,...,1)$.
Then $L$ is self-$\Delta$ concordant to its companion $\Gamma$.
\end{cor}

\begin{rem}\label{remark7}
(1)~Let $L$ be an $n$-component link which is a componentwise satellite
link of type $(\Gamma;l_1,...,l_n)$. Suppose that $\Gamma$ is self-$\Delta$ 
concordant to a trivial link $O$. It is not hard to see that if $\Gamma$ is 
concordant to a link $\Gamma'$, then $L$ is concordant to a link which is a 
componetwise satellite link of type $(\Gamma';l_1,...,l_n)$. This and 
\cite[Proposition~1]{S-Y} imply that $L$ is self-$\Delta$ concordant to a 
link $L'$ which is a componetwise satellite link of type $(O;l_1,...,l_n)$. 
Since each component of $L'$ is separated from the rest components by 
a $2$-sphere, it is $\Delta$-equivalent to the trivial knot \cite{MN}.
This implies that $L'$ is self $\Delta$-equivalent to $O$.
Hence $L$ and $O$ are self-$\Delta$ concordant for any $l_1,...,l_n$.\\
(2)~Let $L$ be a $2$-component link which is a componentwise satellite
link of type $(\Gamma;p,q)$. Then we have that
$\overline{\mu}_L(12)=pq\overline{\mu}_{\Gamma}(12)$ and
$\overline{\mu}_L(1122)=p^2q^2\overline{\mu}_{\Gamma}(1122)$ \cite[Lemma 1]{S-Y}.
Where  $\overline{\mu}(12)$ and $\overline{\mu}(1122)$ are Milnor invariants, 
which are known to be concordance invariants \cite{Casson} and self 
$\Delta$-equivalence invariants \cite{FY}. Suppose that $\Gamma$ is not 
self-$\Delta$ concordant to a trivial link. Then by \cite[Corollary~1.5]{Yasu},
either $\overline{\mu}_{\Gamma}(12)$ or $\overline{\mu}_{\Gamma}(1122)$ is 
nontrivial. Hence if $L$ and $\Gamma$ are self-$\Delta$ concordant, then $|pq|=1$.
\end{rem}

Corollary~\ref{thm:core} implies the following.

\begin{cor}\label{thm:linkcore}
Let $L$ and $L'$ be componentwise satellite links of type
$(\Gamma;\varepsilon_1,...,\varepsilon_n)$ and
$(\Gamma';\varepsilon_1,...,\varepsilon_n)~(\varepsilon_i\in\{-1,1\})$,
respectively. Then $L$ and $L'$ are self-$\Delta$ concordant if and only if
their companions $\Gamma$ and $\Gamma'$ are self-$\Delta$ concordant.
\end{cor}

\begin{rem}
(1)~Let $\Gamma$ be a $2$-component link which is not self-$\Delta$ 
concordant to a trivial link. Let $L$ and $L'$ be componentwise satellite 
links of type $(\Gamma;p,q)$ and $(\Gamma;p',q')$, respectively.
By Remark~\ref{remark7} (2), if $L$ and $L'$ are self-$\Delta$ concordant,
then $|pq|=|p'q'|$. \\
(2)~In Example~\ref{example4}, we show that for any $p~(|p|\geq 2)$, there 
are two links $L$ and $L'$ that are not self-$\Delta$ concordant, but are 
both componentwise satellite links of type $(H;1,p)$ for the Hopf link $H$.
\end{rem}

\section{Proof of Theorem~\ref{free-homotopy}}

In order to prove Theorem~\ref{free-homotopy}, we need the following 
lemma which is a direct corollary of \cite[Theorem~5.6]{MKS}.

\begin{lem}\label{trivial} $($\cite[Theorem~5.6]{MKS}$)$
Let $F(r)=\langle x_1,...,x_r\rangle$ be the free group of rank $r$.
An element $w\in F(r)$ is trivial if and only if the Magnus expansion 
$E(w)$ of $w$ is equal to $1$.
\end{lem}

Although the lemma above follows from \cite[Theorem~5.6]{MKS}, the proof 
is very short, and so we include it here for the reader's convenience.

\begin{proof}
The \lq only if' part is obvious. We show \lq if' part. The proof is
essentially the same as the proof of \cite[Theorem~5.6]{MKS}.

Let $w=x_{i_1}^{p_1}\cdots x_{i_s}^{p_s}$ be a freely reduced word which 
represents a nontrivial element, where $p_j$ are non-zero integers and 
$1\leq i_k\neq i_{k+1}\leq r$. It is not hard to see that for any $i$ 
and $p$ \[E(x_i^p)=1+p X_i + X_i^2f_i,\] where $f_i$ is an infinite 
power series in $X_i$. This implies that 
\[E(w)=(1+p_1 X_{i_1}+X_{i_1}^2 f_{i_1})
\cdots (1+p_s X_{i_s}+X_{i_s}^2f_{i_s}).\] 
Since $1\leq i_k\neq i_{k+1}\leq r$, the coefficient of
$X_{i_1}\cdots X_{i_s}$ is $p_1\cdots p_s(\neq 0)$. 
Hence $E(w)\neq 1$. This completes the proof.
\end{proof}

\begin{proof}[Proof of Theorem~\ref{free-homotopy}]
First we show the \lq only if' part.
Suppose that $K_0$ is null-homotopic in $S^3\setminus(L-K_0)$.
Let $L'$ be a link obtained from $L$
by taking a number of zero-framed parallels of $K_i~(i=1,...,n)$.
Then $K_0$ is also null-homotopic in $S^3\setminus(L'- K_0)$.
In particular, $L'$ is link-homotopic to a trivial link.
Hence all Milnor's link-homotopy invariants of $L'$ vanish.
By \cite[Theorem~7]{Milnor2}, $\overline{\mu}_L(I0)=0$ 
for any multi-index $I$ with entries from $\{1,...,n\}$.

Now we show \lq if' part.
Set $G(L)=\pi_1(S^3-L)$ and $G_q(L)~(q\geq 1)$ the $q$th lower central 
subgroup of $G(L)$. There is the natural homomorphism from $G(L)/G_q(L)$ 
to $G(L-K_0)/G_q(L-K_0)$ so that the $i$th meridians $m_{i}~(i=1,...,n)$ 
of $L$ map to the $i$th meridians $m'_{i}$ of $L-K_0$, and the $0$th 
meridian $m_{0}$ maps to the trivial element $1$.
Let $l$ be the $0$th longitude of $L$.
Then $l$ is written as a word $w_l(m_0,m_1,...,m_n)$ in $G(L)/G_q(L)$ 
and a word $w_l(m'_1,...,m'_n)$ in $G(L-K_0)/G_q(L-K_0)$.
We note that $w_l(1,m_1,...,m_n)$ sends to $w_l(m'_1,...,m'_n)$ via 
the homomorphism above.

The Magnus expansion $E(w_l(1,m_1,...,m_n))$ 
can be obtained from the expansion
\[E(w_l(m_0,m_1,...,m_n))=1+\sum\mu_{L}(h_1...h_s 0)X_{h_1}\cdots X_{h_s}\]
by substituting $0$ for $X_0$.
Hence by the assumption that $\ov{\mu}_L(I0)=0$ for any multi-index 
$I$ with entries from $\{1,...,n\}$, we have
\[E(w_l(1,m_1,...,m_n))=E(w_l(m'_1,...,m'_n))=1.\]
Since $G(L-K_0)$ is a free group, by Lemma~\ref{trivial}, 
$l$ is trivial in $G(L-K_0)$.
\end{proof}

\section{Proof of Theorem~\ref{free-delta}}


Let $L=K_1\cup\cdots\cup K_n$ be an $n$-component link in a $3$-manifold 
$M$ and $B\subset M$ a band attaching a single component $K_i$ with 
coherent orientation, i.e., $B\cap L=K_i\cap B\subset \partial B$ 
consists of two arcs whose orientations from $K_i$ are opposite to those 
from $\partial B$. Then $L'=(L\cup\partial B)-\mathrm{int}(B\cap K_i)$, 
which is an $(n+1)$-component link, is said to be obtained from $L$ by 
{\em fission}  (along a band $B$) in $M$, and conversely $L$ is said to 
be obtained from $L'$ by {\em fusion} (along a band $B$) in $M$ \cite{KSS}.

The following lemma is shown in \cite{Yasu}.

\begin{lem}\label{reorder} $($\cite[Lemma~3.5]{Yasu}$)$
Let $L_1, L_2, L_3$ be links such that $L_2$ is obtained from $L_1$ by 
a single fission, and that $L_3$ is obtained from $L_2$ by a single 
self $\Delta$-move. Then there is a link $L_2'$ such that $L_2'$ is 
obtained from $L_1$ by a single self $\Delta$-move, and that $L_3$ is 
obtained from $L'_2$ by a single fission. Here we call a $\Delta$-move 
a self $\Delta$-move if the three strands belong to a link obtained from 
a single component by fission.
\end{lem}

The proof of the following lemma is an easy modification of
the proof of \cite[Theorem]{Shi} (or \cite[Theorem~2]{NSY}).

\begin{lem}\label{ribbon}
Let $K_0\cup K_1\cup\cdots\cup K_n$ be an $(n+1)$-component link.
If $K_0$ bounds a ribbon disk $($a singular disk with only ribbon 
singularities$)$ in $S^3\setminus(L-K_0)$, then $K_0$ is 
$\Delta$-equivalent to the trivial knot in $S^3\setminus(L-K_0)$.
\end{lem}

Now we are ready to prove Theorem~\ref{free-delta}. 
The proof is given by combining Corollary~\ref{free-homotopy2}, 
and Lemmas~\ref{reorder} and \ref{ribbon}.

\begin{proof}[Proof of Theorem~\ref{free-delta}]
Let $F_0\cup F_1\cup\cdots\cup F_n$ be a disjoint union of
orientable surfaces with $\partial F_i=K_i~(i=0,1,...,n)$ and
$F_i\cap F_j=\emptyset~(i\neq j)$. Let $G$ be a bouquet graph 
which is a spine of $F_0$, i.e., $G$ consists of $2g$ loops
$C_1,...,C_{2g}$ and a point $P$ with $C_i\cap C_j=P~(i\neq j)$, 
and $G$ is a deformation retract of $F_0$, where $g$ is the genus 
of $F_0$. We may assume that $F_0$ consists of a disk $D$ and bands
$b_1,...,b_{2g}$ so that $D$ contains $P$ and $b_i\cup D$ is an annulus 
with the core $C_i$ for each $i$. By Corollary~\ref{free-homotopy2}, 
each $C_i$ is homotopic to $P$ in $S^3\setminus(L-K_0)$. Hence $G$ is 
homotopic to $P$ in $S^3\setminus(L-K_0)$ with $P$ fixed. This implies 
that $F_0$ can be transformed into a surface $F'_0$ that is contained in 
a $3$-ball $B^3\subset S^3\setminus(L-K_0)$ by {\em band-pass moves} 
between $b_i$ and $b_j~(1\leq i\leq j\leq 2g)$ as illustrated in
Figure~\ref{band-pass}. Therefore $\partial F_0=K_0$ can be transformed
into an {\em algebraically split} link $L_0$ in $B^3$ by a finite
sequence of fissions as illustrated in Figure~\ref{fission}, where a
link is algebraically split if the linking numbers of its all
$2$-component sublinks vanish. Hence $L_0$ is $\Delta$-equivalent to a
trivial link in $B^3$ \cite{MN}. It follows from Lemma~\ref{reorder}
that there is a knot $K'_0$ such that $K'_0$ is $\Delta$-equivalent to 
$K_0$ in $S^3\setminus(L- K_0)$ and is transformed into a trivial link by 
a finite sequence of fissions in $S^3\setminus(L- K_0)$. We note that 
$K'_0$ is a ribbon knot and $K'_0$ bounds a ribbon disk in 
$S^3\setminus(L-K_0)$. This and Lemma~\ref{ribbon} imply that $K'_0$ 
is $\Delta$-equivalent to the trivial knot in $S^3\setminus(L-K_0)$.
This completes the proof.
\end{proof}
\vspace{-7mm}
\begin{figure}[!h]
\includegraphics[trim=0mm 0mm 0mm 0mm, width=.4\linewidth]{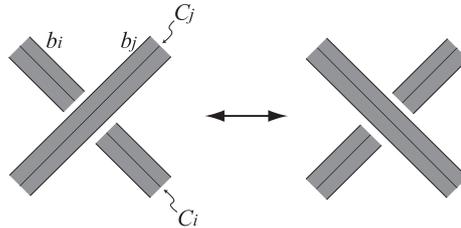} 
\caption{Band-pass moves between $b_i$ and $b_j$}\label{band-pass}
\end{figure}
\vspace{-3mm}
\begin{figure}[!h]
\includegraphics[trim=0mm 0mm 0mm 0mm, width=.4\linewidth]{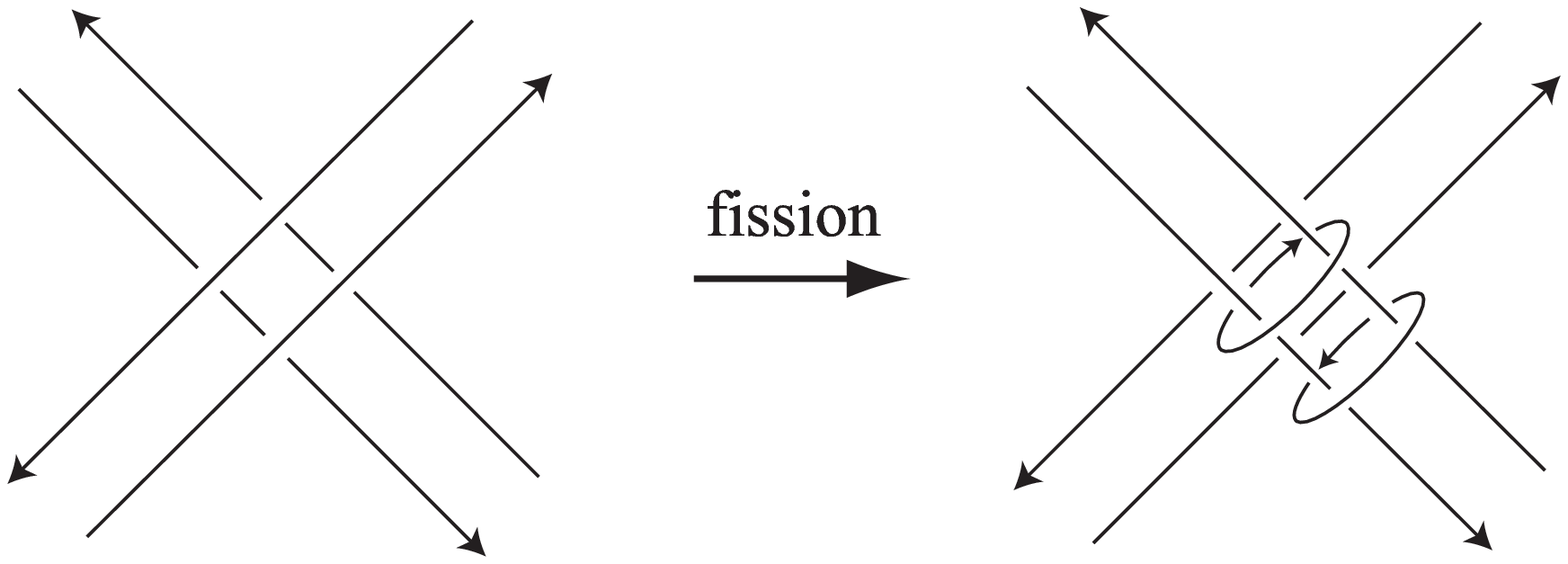}
\caption{}\label{fission}
\end{figure}

\section{Proof of Theorem~\ref{homotopy-delta}}

Habiro \cite{H} and Goussarov \cite{G} independently introduced the
notion of a $C_k$-move. A $C_k$-move is a local move on links as
illustrated in Figure \ref{Ck-move}, which can be regarded as a kind
of `higher order crossing change'. In particular, a $C_1$-move is a
crossing change and a $C_2$-move is a $\Delta$-move. We call a
$C_k$-move a {\em self $C_k$-move} if all the strands belong to the
same component of a link. The (self) $C_k$-move generates an
equivalence relation on links, called \emph{$($self$)$
$C_k$-equivalence}, which becomes finer as $k$ increases. This
notion can also be defined by using the theory of claspers \cite{H}.

\begin{figure}[!h]
\includegraphics[trim=0mm 0mm 0mm 0mm, width=.75\linewidth]{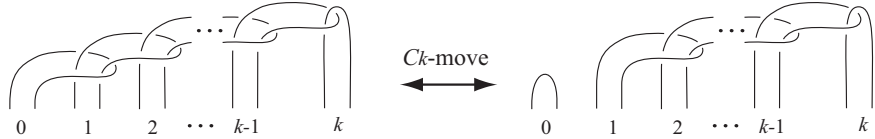}
\caption{A $C_k$-move involves $k+1$ strands of a link, labelled here 
with the integers from $0$ to $k$.}\label{Ck-move}
\end{figure}

The first and the last authors \cite{FY} showed that any Milnor invariant
$\overline{\mu}(I)$ with $r(I)\leq k$ is a self $C_k$-equivalence invariant.
The proof of \cite[Theorem~1.1]{FY} implies the following proposition.
Note that this proposition is a generalization of the \lq only if' part of
Theorem~\ref{free-homotopy}.

\begin{prop}\label{free-Ck-inv}
Let $L$ be an $n$-component link.
If the $i$th component $K$ is $C_k$-equivalent to the trivial knot in 
$S^3\setminus(L-K)$, then $\overline{\mu}_L(I)=0$ for any multi-index $I$ 
with entries from $\{1,...,n\}$ such that the index $i$ appears in $I$ at 
least once and at most $k$ times.
\end{prop}

The \lq only if' part of Theorem~\ref{homotopy-delta} holds for more 
general setting as follows. Let $W^{i}(L)$ be the link obtained from 
$L$ by Whitehead doubling the $i$th component of $L$.

\begin{prop}\label{whitehead}
Let $L=K_0\cup K_1\cup\cdots\cup K_n$ be an $(n+1)$-component link.
If $K_0$ is null-homotopic in $S^3\setminus(L-K_0)$, then $W^0(K_0)$ is  
$\Delta$-equivalent to the trivial knot in $S^3\setminus(L-K_0)$.
\end{prop}

\begin{proof}
Let $K_0'$ be a knot obtained from $K_0$ by a single crossing change in
$S^3\setminus(L-K_0)$. Then $W^0(K'_0)$ is obtained from $W^0(K_0)$ by 
a local move as illustrated in Figure~\ref{delta-pass}, which is realized 
by $\Delta$-move (for example see \cite{TY}) in $S^3\setminus(L-K_0)$.
It follows that $W^0(K_0)$ is $\Delta$-equivalent to a Whitehead doubled
trivial knot, which is also trivial, in $S^3\setminus(L-K_0)$.
This completes the proof.
\end{proof}

\begin{figure}[!h]
\includegraphics[trim=0mm 0mm 0mm 0mm, width=.4\linewidth]{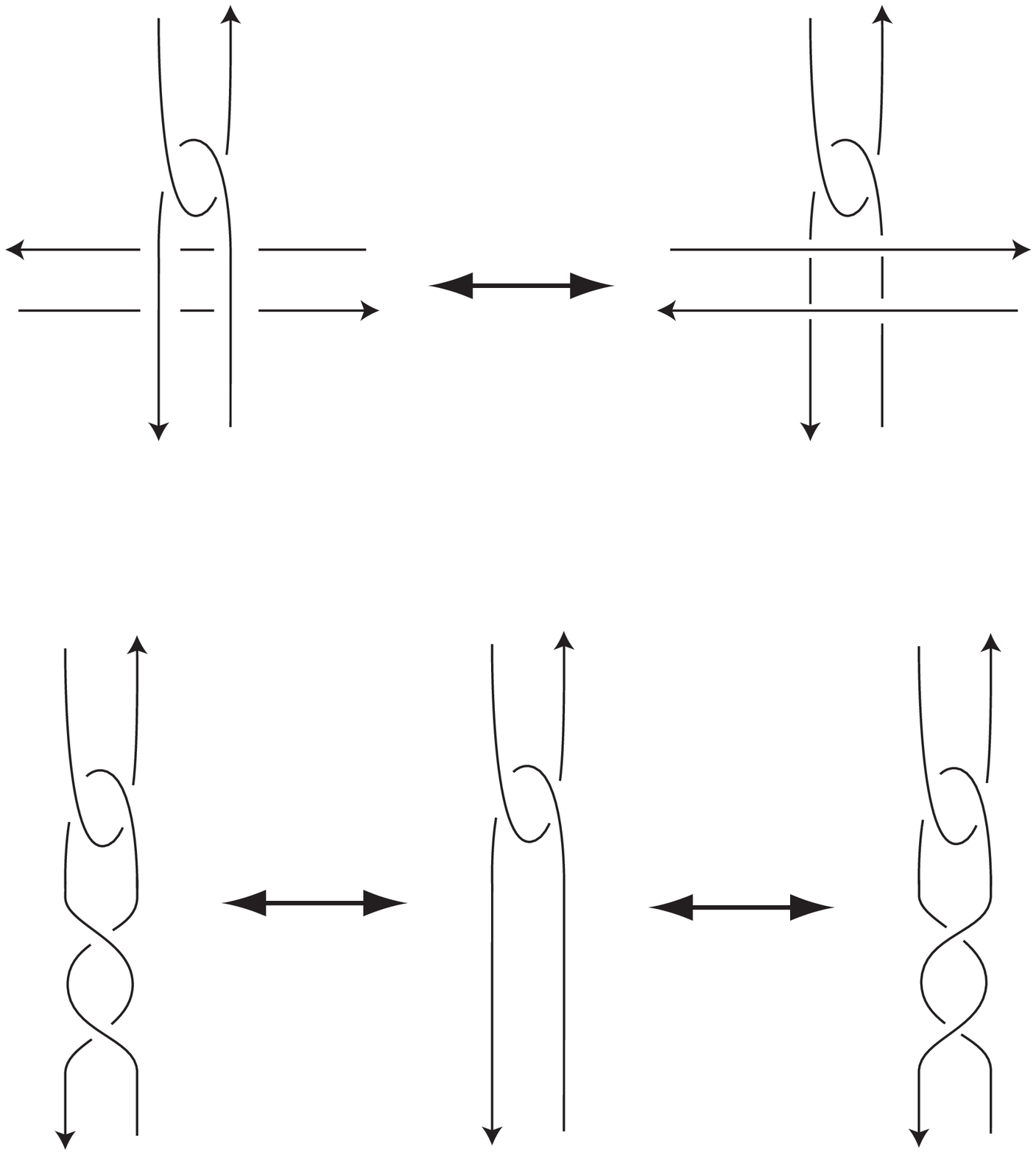}
\caption{}\label{delta-pass}
\end{figure}

\begin{proof}[Proof of Theorem~\ref{homotopy-delta}]
The \lq only if' part follows from Proposition~\ref{whitehead}.

We show the \lq if' part.
Suppose that $K_0$ is not null-homotopic in $S^3\setminus(L-K_0)$. Then, 
by Theorem~\ref{free-homotopy}, there is a multi-index $I$ with entries 
from $\{1,...,n\}$ such that $\overline{\mu}_L(I0)\neq 0$. This and 
\cite[Theorem~1.1]{MY} imply that $\overline{\mu}_{W^0(L)}(II00)\neq 0$.
Proposition~\ref{free-Ck-inv} completes the proof.
\end{proof}

\section{Proof of Theorem~\ref{delta-concordance}}

Theorem~\ref{delta-concordance} follows from the proposition below.

\begin{prop}\label{prp:fission}
Let $K$ be a knot in a solid torus $V \subset S^3$ with a 
meridian disk $M$ such that $K$ intersects $M$ transversely. 
Assume that $\mathrm{lk}(\partial M,K)=p \neq 0$ and that 
$|M\cap K|=|p|+2q~(q>0)$. Then by performing $(|p|+q)$ 
fissions in $V$, $K$ can be transformed into $L_1 \cup L_2$
that satisfies the following: $L_1$ is $p$ zero-framed parallels 
of the core $c$ of $V$, and $L_2$ is an algebraically split link 
with $(q+1)$-components in a $3$-ball in $V-L_1$. The curves in 
$L_1$ have orientation consistent with $V$ if $p$ is positive, 
and the opposite orientation if $p$ is negative.
\end{prop}

In order to prove Proposition~\ref{prp:fission}, we need the following lemma.

\begin{lem}\label{lem:n-fission}
Let $K$ and $M$ be as in Proposition~\ref{prp:fission}. 
There is a sequence of $q$ fissions that transforms $K$ 
into an algebraically split link $K' \cup L'$ such that $K'$ 
is a knot with $|\mathrm{lk}(\partial M,K')|=|M\cap K'|=|p|$ 
and $L'$ is a $q$-component link in $V-M$.
\end{lem}

\begin{proof}
First, we inductively transform $K$ into a link $K^q \cup L^q$, 
which is not necessarily algebraically split, such that $L^q$ is 
contained in $V-M$ and $|\mathrm{lk}(\partial M, K^q)|=|M\cap K^q|=|p|.$

[{\bf 1st Step}]
Choose two points $a_1$ and $b_1$ in $M\cap K$ so that \\
(1)~$\text{sign} (a_1)=1$, $\text{sign} (b_1)=-1$ and \\
(2)~there is a subarc $\alpha_1$ in $K$ with 
$M\cap \alpha_1=\partial \alpha_1=\{a_1,b_1\}$ such that the orientation 
from $a_1$ to $b_1$ along $\alpha_1$ is as same as that of $K$. \\
Let $\gamma_1$ be an arc in $M$ with 
$\gamma_1\cap K=\partial \gamma_1=\{a_1,b_1\}$, and let $N(\gamma_1)$ 
be a fission band of $K$ which is an $I$-bundle over $\gamma_1$ with 
$N(\gamma_1)\cap M=\gamma_1$. By fission along $N(\gamma_1)$, we have 
a new link $K^1\cup K^{(1)}$ from $K$, where $K^1\cap\alpha_1=\emptyset$.
Note that $M\cap (K^1\cup K^{(1)})=M\cap K^1$, see Figure~\ref{fission2}.

\begin{figure}[!h]
\includegraphics[trim=0mm 0mm 0mm 0mm, width=.6\linewidth]{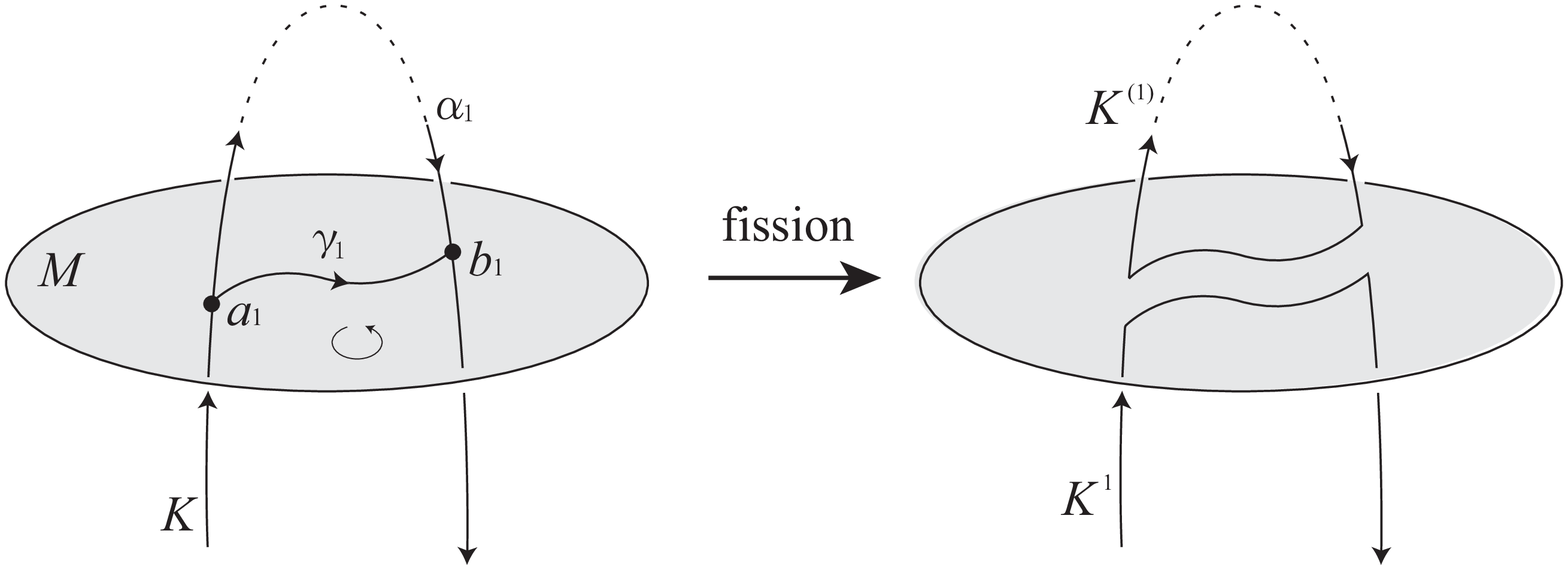}
\caption{}\label{fission2}
\end{figure}

[{\bf 2nd Step}]
Choose two points $a_2$ and $b_2$ in $M\cap K^1$ so that \\
(1)~$\text{sign} (a_2)=1$, $\text{sign} (b_2)=-1$ and \\
(2)~there is a subarc $\alpha_2$ in $K^1$ with 
$M\cap \alpha_2=\partial \alpha_2=\{a_2,b_2\}$ such that the orientation 
from $a_2$ to $b_2$ along $\alpha_2$ is as same as that of $K^1$. \\
Let $\gamma_2$ be an arc in $M$ with 
$\gamma_2\cap K^1=\partial \gamma_2=\{a_2,b_2\}$, and let $N(\gamma_2)$ 
be a fission band of $K^1$ which is an $I$-bundle over $\gamma_2$ with 
$N(\gamma_2)\cap M=\gamma_2$. By fission along $N(\gamma_2)$, we have 
a new link $K^2\cup K^{(1)}\cup K^{(2)}$ from $K^1\cup K^{(1)}$, where 
$K^2\cap\alpha_2=\emptyset$.

Running this process until the $q$-th step, we have 
$K^q \cup L^q=K^q \cup (K^{(1)} \cup\cdots\cup K^{(q)})$ with 
$M\cdot(K^q\cup L^q)=M\cdot K^q=\mathrm{lk}(\partial M,K^q) =\mathrm{lk}
(\partial M,K)$. From the construction, $L^q$ is a $q$-component link in
$V-M$. Now we show that we can choose $\gamma_1,...,\gamma_q$ so that 
$K^q\cup L^q$ is an algebraically split link.

Set $K^q=K^{(q+1)}$ and 
$l_{i,j}=|\mathrm{lk}(K^{(i)},K^{(j)})|~(1\leq i<j\leq q+1)$. Then 
we have a vector \[(l_{1,2},l_{1,3},...,l_{1,q+1},
l_{2,3},l_{2,4},...,l_{2,q+1},...,l_{q-1,q},l_{q-1,q+1},l_{q,q+1}).\]
This vector depends on the choice of $\gamma_1,...,\gamma_q$.
We denote the vector by $v(\gamma_1,...,\gamma_q)$. We choose arcs 
$\gamma_1,...,\gamma_q$ so that $v(\gamma_1,...,\gamma_q)$ is the 
minimum under the lexicographic order. If $v(\gamma_1,...,\gamma_q)$ 
is a non-zero vector, then we have that $l_{i,j}\neq 0$ for some 
$1\leq i<j\leq q+1$.

Case 1: When $i\neq q$ and $\mathrm{lk}(K^{(i)},K^{(j)})>0$ (resp. $<0$),
we choose a disk $D_j$ which is a regular neighborhood of $a_j$ in $M$ 
with $\mathrm{lk}(\partial D_j,K)=1$ (resp. $=-1$). Let $B$ be a band 
attached to both $\partial D_j$ and $\gamma_i$ with coherent orientation, 
see Figure~\ref{band}.

\begin{figure}[!h]
\includegraphics[trim=0mm 0mm 0mm 0mm, width=.7\linewidth]{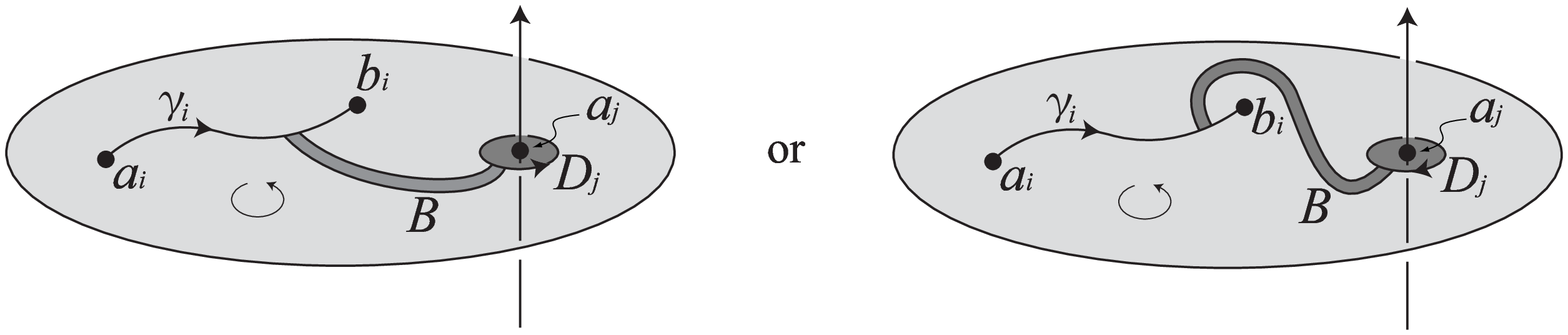}
\caption{}\label{band}
\end{figure}

We may assume that $(D_j\cup B)\cap K=D_j\cap K=a_j$. Let 
$\gamma_i'=\gamma_i\cup\partial (B\cup D_j)-\mathrm{int}(\gamma_i\cap B)$
be an arc obtained from $\gamma_i\cup\partial D_j$ by fission along $B$.
For $\gamma_1,...,\gamma_{i-1},\gamma'_i,\gamma_{i+1},...,\gamma_q$, 
we have a new vector 
$v(\gamma_1,...,\gamma_{i-1},\gamma'_i,\gamma_{i+1},...,\gamma_q)=
(l'_{1,2},...,l'_{q,q+1})$. By the construction of $\gamma'_i$, we note 
that $l'_{i,j}=l_{i,j}-1$ and that if $l'_{s,t}\neq l_{s,t}$, then 
$s\geq i$ and $t\geq j$. This contradicts the minimality of the choice of 
$\gamma_1,...,\gamma_n$.

Case 2: When $i=q$.
Let $a_{n+1}$ be a point in $K^{q}\cap M$. Then by arguments similar to 
that in Case 1, we also have a contradiction.
\end{proof}

\begin{proof}[Proof of Proposition \ref{prp:fission}]
Let $K' \cup L'$ be a link as in Lemma \ref{lem:n-fission}. Push the
$3$-ball $V-\mathrm{int} N(M)$ into the interior of $V$ and let the
result be $B^3$. Then $K'\cap (\overline{V-B^3})$ consists $|p|$ arcs
$\{c_1,...,c_{|p|}\} \times [0,1]$, where $\{c_1,...,c_{|p|}\}=K'\cap M$. 
Then we can take $|p|$-bands in $V-B^3$ so that fission along the 
$|p|$-bands transforms $K' \cup L'$ into the union of the $p$ zero-framed 
parallels $L_1$ of the core of $V$ and the link $L_2$ with $(q+1)$-components 
in $B^3$. Since $L_2$ is an algebraically split link, $L_1 \cup L_2$ 
is the required link in the proposition.
\end{proof}

\begin{proof}[Proof of Theorem \ref{delta-concordance}]
Let $K$ and $K'$ be knots in a solid torus $V \subset S^3$, 
which is the complement $V$ of the trivial knot $O$, with 
$\mathrm{lk}(\partial M,K)=\mathrm{lk}(\partial M,K')=1$, 
where $M$ is a meridian disk of $V$ with $\partial M=O$.

Suppose that $K$ intersects $M$ transversely and $|M\cap K|=1+2q$.
From Proposition \ref{prp:fission}, there are $(1+q)$ fissions in
$V$ which transform $K$ into $L_1 \cup L_2$ such that $L_1$ is the
core of $V$ and $L_2$ is an algebraically split link with $q$
components in a $3$-ball $B^3$ in $V-L_1$. Since an algebraically
split link is $\Delta$-equivalent to a trivial link \cite{MN}, $L_2$
is $\Delta$-equivalent to a trivial link in $B^3$.
This implies that $K$ can be transformed into a link $L_1 \cup L_2$
by a finite number of fissions, and $L_1 \cup L_2$ into a split sum
of $L_1$ and a trivial link by self $\Delta$-moves. (Recall that a
self $\Delta$-move means a $\Delta$-move whose three strands belong
to a link obtained from a single component by fissions.) By
Lemma~\ref{reorder}, there is a knot $K''$ such that $K$ is self
$\Delta$-equivalent to $K''$ and $K''$ is concordant to $L_1$.

By a similar argument, $K'$ is $\Delta$ concordant to $L_1$ and
hence $\Delta$ concordant to $K$.
\end{proof}

\section{Examples}

\begin{example}\label{example1}
Let $L=K_1\cup K_2\cup K_3$ be the closure of the 3-string link as
illustrated in Figure~\ref{example-fig1}, which is represented as a
trivial string link with claspers. Roughly speaking, each clasper can 
be replaced with a tangle as illustrated in Figure~\ref{milnor-tangle}. 
For a precise definition, see \cite{H}. Note that $L$ is a Brunnian
link. By using the calculation method described in \cite[Remark~5.3]{Yasu2}, 
we have $\overline{\mu}_L(I)=0$ for any $I$ with $|I|\leq 3$, and
$|\overline{\mu}_L(3213)|=|\overline{\mu}_L(1231)|=1$. In particular, 
$\overline{\mu}_L(I)=0$ for any $I$ with $r(I)=1$, hence $L$ is 
link-homotopic to a trivial link. Since $\overline{\mu}$ has 
\lq cyclic symmetry' \cite[Theorem~8]{Milnor2}, 
$|\overline{\mu}_L(3321)|=|\overline{\mu}_L(1332)|=|\overline{\mu}_L(1123)|=1$.
It follows from Corollary~\ref{brunnian-free-homotopy}  that any
component $K_i$ is   not null-homotopic in $S^3\setminus(L-K_i)~(i=1,2,3)$.
\end{example}

\begin{figure}[!h]
\includegraphics[trim=0mm 0mm 0mm 0mm, width=.23\linewidth]{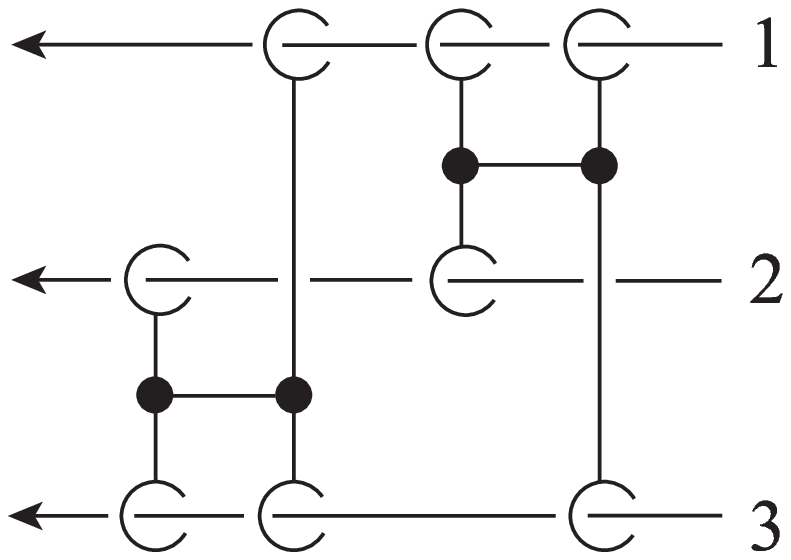}
\caption{}\label{example-fig1}
\end{figure}

\begin{figure}[!h]
\includegraphics[trim=0mm 0mm 0mm 0mm, width=.7\linewidth]{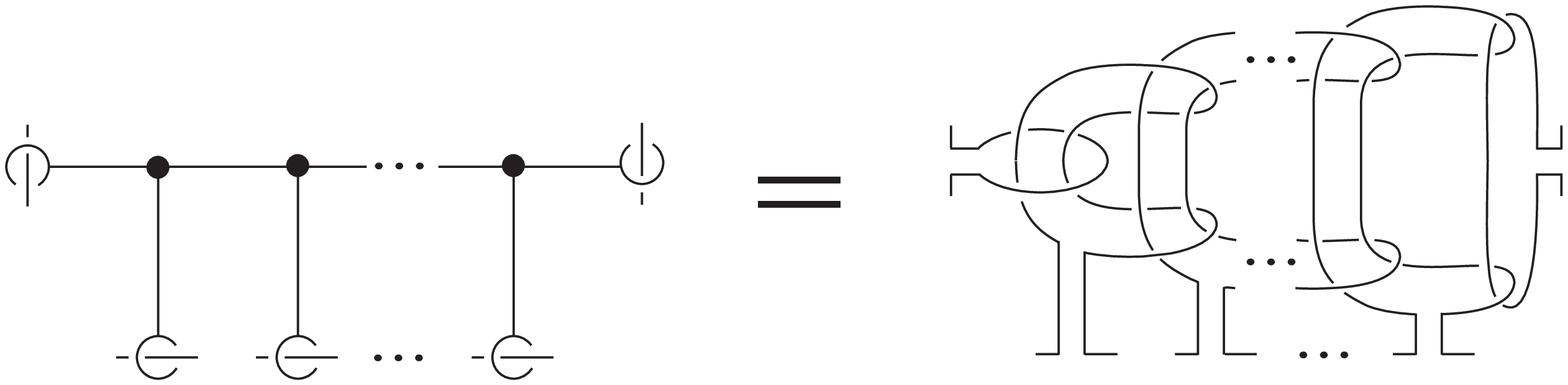}
\caption{}\label{milnor-tangle}
\end{figure}

\begin{example}\label{example2}
Let $L=K_1\cup K_2$ be the closure of the $2$-string link illustrated 
in Figure~\ref{example-fig2}. Note that $L$ is a Brunnian link. Then, 
by using the calculation method described in \cite[Remark~5.3]{Yasu2}, 
we have $\overline{\mu}_L(I)=0$ for any $I$ with $|I|\leq 5$, and
$|\overline{\mu}_L(222211)|=|\overline{\mu}_L(111122)|= 2$. It follows 
from \cite[Corollary~1.5]{Yasu2} and Proposition~\ref{free-Ck-inv}
that $L$ is self $\Delta$-equivalent to a trivial link and 
any component $K_i$ is  not $\Delta$-equivalent to a trivial knot in 
$S^3\setminus(L-K_i)~(i=1,2)$. 
In contrast, we notice by Remark~\ref{remark2} that each component 
$K_i$ is null-homotopic in $S^3\setminus(L-K_i)$.

\begin{figure}[!h]
\includegraphics[trim=0mm 0mm 0mm 0mm, width=.4\linewidth]{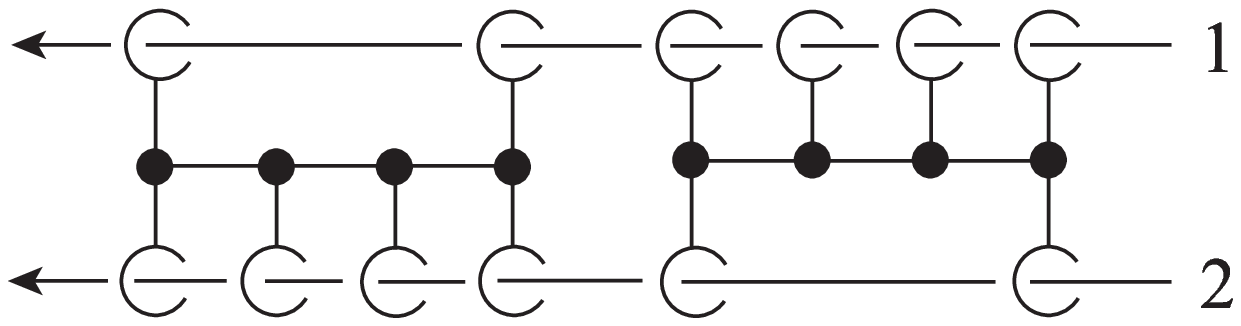}
\caption{}\label{example-fig2}
\end{figure}

\end{example}

For any $k\geq 2$, there are knots that are $C_k$-equivalent to the
trivial knot and not $C_{k+1}$-equivalent to the trivial knot
\cite{O}. Let $L$ be a link which is a split sum of such knots. Then
each component $K$ of $L$ is $C_k$-equivalent to the trivial link in
$S^3\setminus (L-K)$ and is not $C_{k+1}$-equivalent to the trivial
link in $S^3\setminus (L-K)$. It seems to be uninteresting. Hence we
show that for each $k\geq 2$, there is a Brunnian 2-component link
$L$ such that each component $K$ of $L$ is $C_{k-1}$-equivalent to the
trivial knot and is not $C_{k}$-equivalent to the trivial knot in
$S^3\setminus (L-K)$.

\begin{example}\label{example3}
Let $L_k~(k\geq 2)$ be the 2-component link as illustrated in
Figure~\ref{example-fig3}. Then each component of $L_k$ is not
$C_{k}$-equivalent to the trivial knot in the complement of the
other component, but is $C_{k-1}$-equivalent to the trivial knot in
the complement of the other component.
\end{example}

\begin{figure}[!h]
\includegraphics[trim=0mm 0mm 0mm 0mm, width=.45\linewidth]{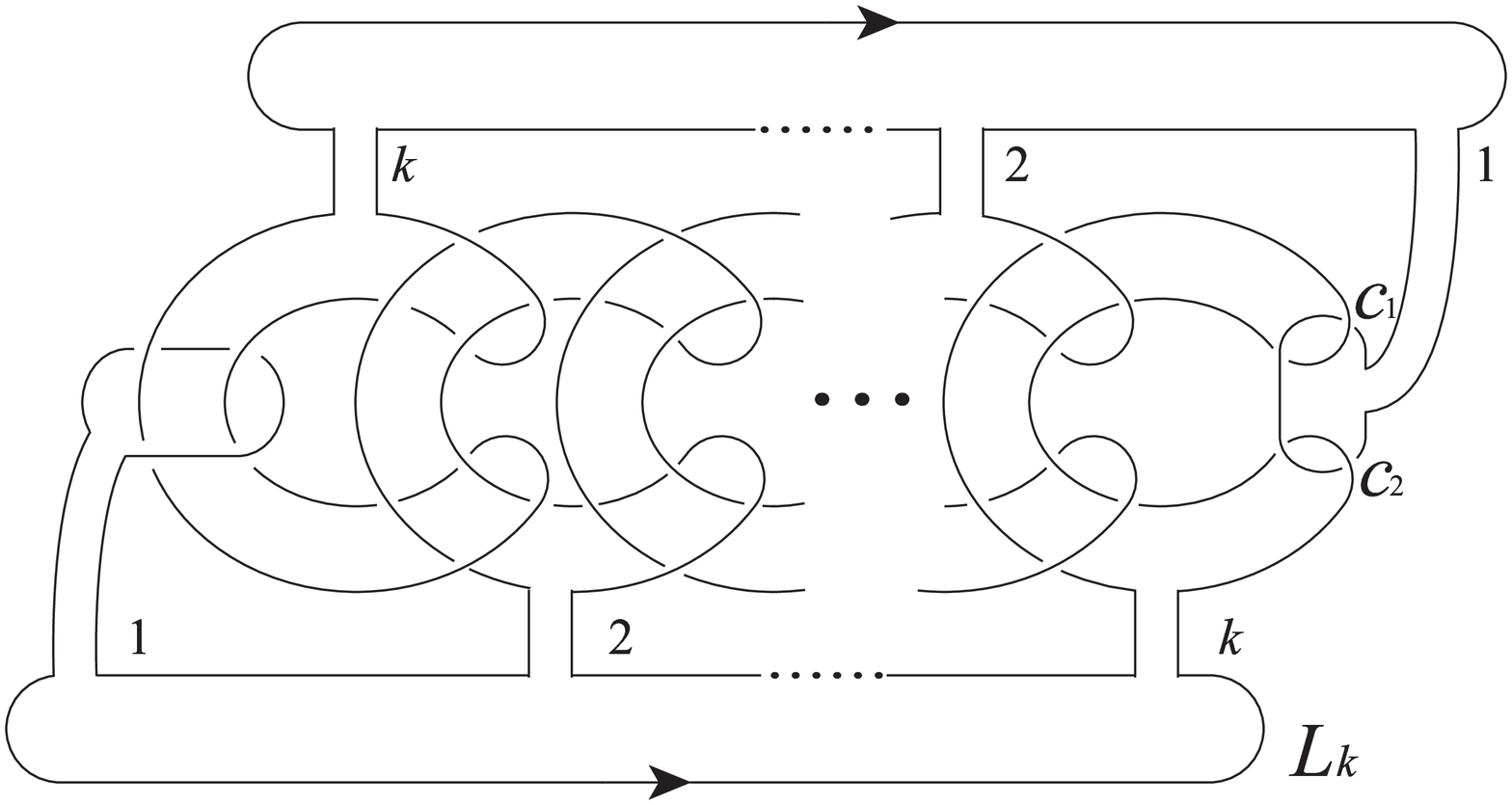}
\caption{}\label{example-fig3}
\end{figure}

\begin{rem}
In the proof of Example~\ref{example3}, we show that
$\overline{\mu}_{L_k}([p,q])=0$ for any $p,q~(p+q\leq 2k,~p\neq q)$ 
and $\overline{\mu}_{L_k}([k,k])=-2$, where $\overline{\mu}([p,q])$ 
denotes $\overline{\mu}(11...122...2)$ with $1$ appearing $p$ times 
and $2$ appearing $q$ times.
\end{rem}

\begin{proof}
First we compute the Conway polynomial $\nabla_{L_k}(z)$ mod $z^{2k}$.
By changing/splicing the two crossings $c_1$ and $c_2$ in 
Figure~\ref{example-fig3}, we have
\[\nabla_{L_k}=\nabla_H(z)-z\nabla_{K_k}-z^2\nabla_{L'_k},\] where
$H$ is the Hopf link with $\nabla_H(z)=z$, $K_k$ is the knot as 
illustrated in Figure~\ref{example-fig3-1} and $L'_k$ is the link 
as illustrated in Figure~\ref{example-fig3-2}.

\begin{figure}[!h]
\includegraphics[trim=0mm 0mm 0mm 0mm, width=.45\linewidth]{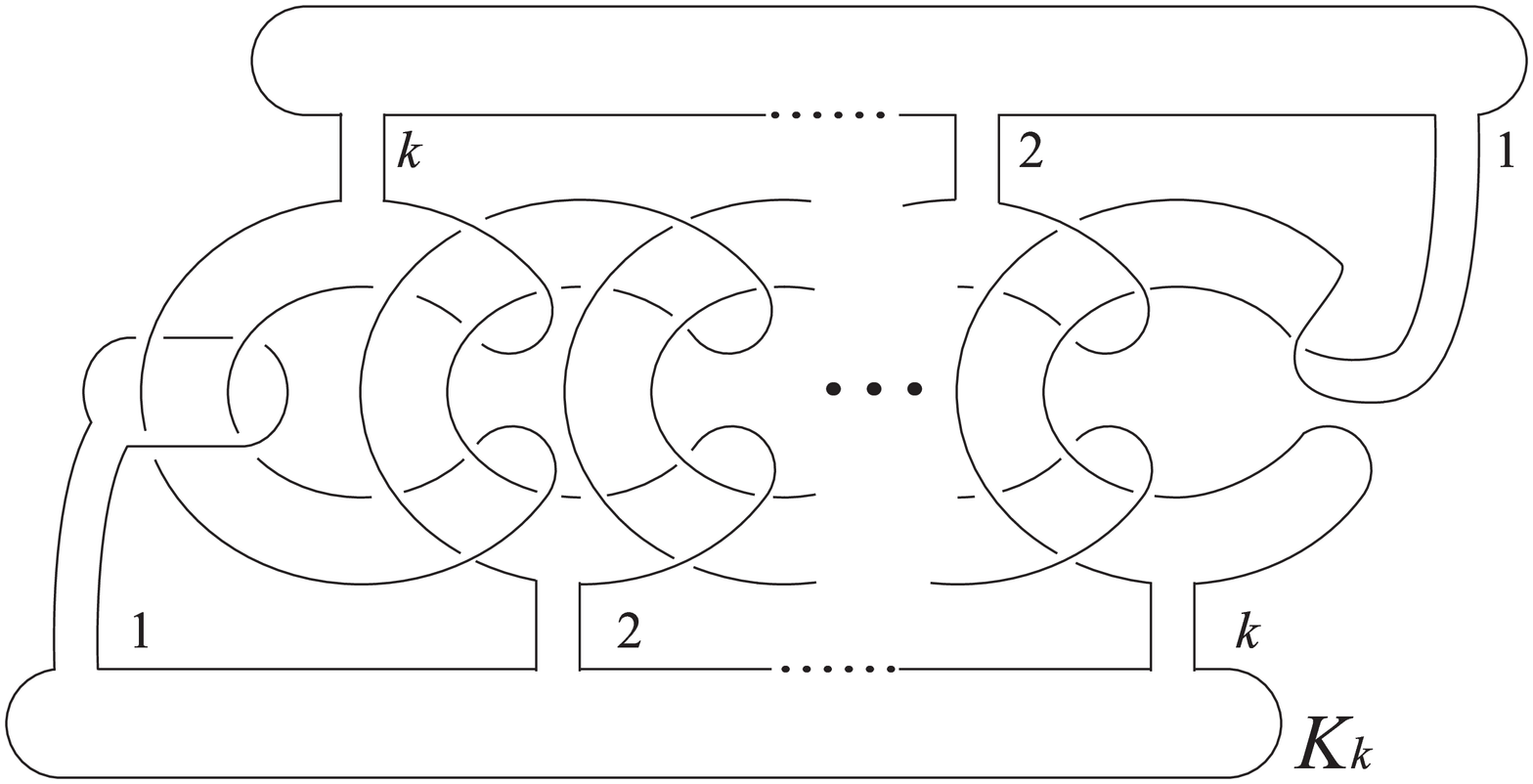}
\caption{}\label{example-fig3-1}
\end{figure}

\begin{figure}[!h]
\includegraphics[trim=0mm 0mm 0mm 0mm, width=.45\linewidth]{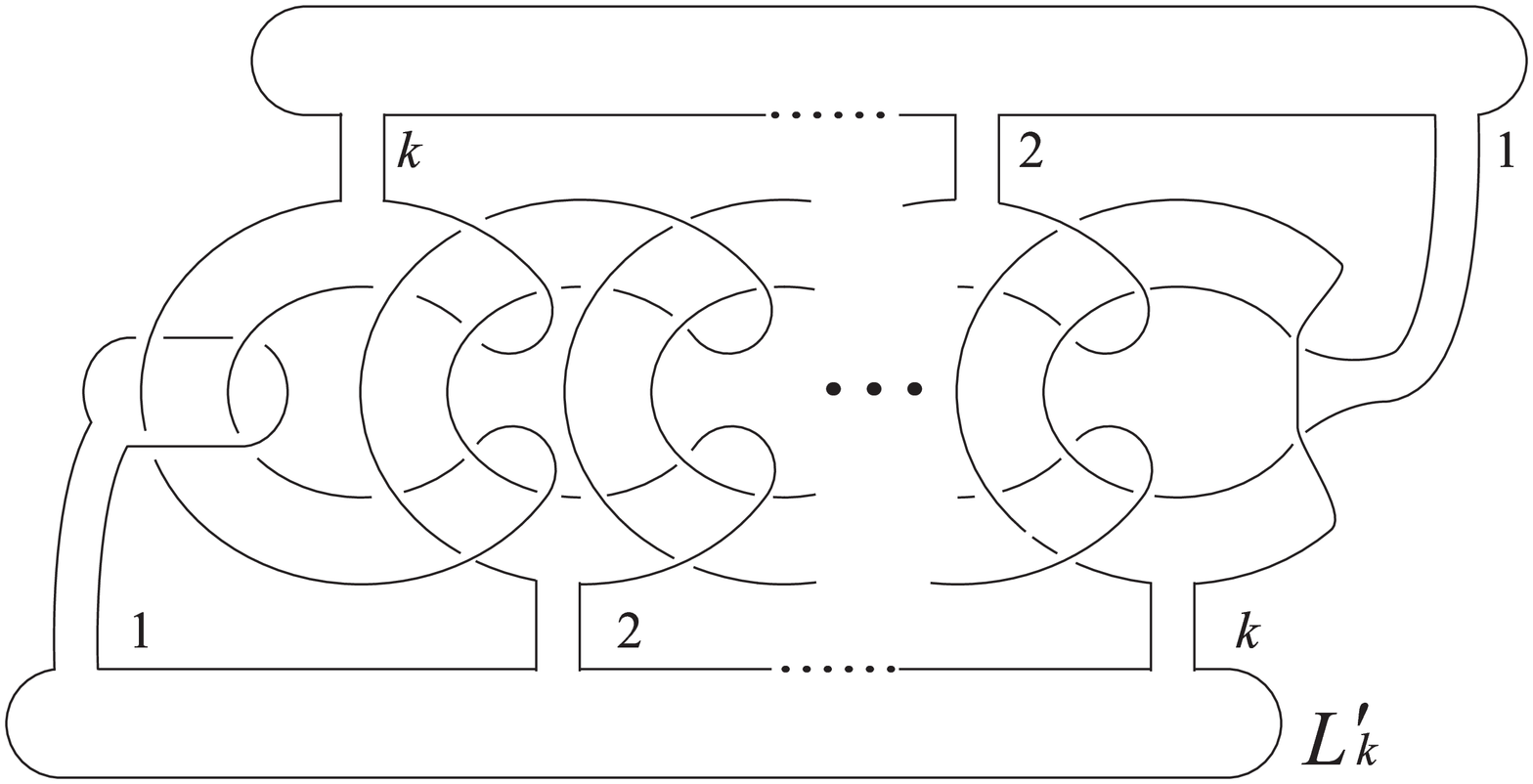}
\caption{}\label{example-fig3-2}
\end{figure}

Note that $L'_k$ is $C_{2k-2}$-equivalent to a trivial link. Since
the finite type invariants of order $\leq m-1$ are invariants for
$C_m$-equivalence \cite{H}, and since the $z^{m-1}$-coefficient
$a_{m-1}$ of the Conway polynomial is a finite type invariant of
order $\leq m-1$ \cite{BNv}, we have $\nabla_{L'_k}(z)\equiv 0$ mod
$z^{2k-2}$. Hence we have $\nabla_{L_k}(z)\equiv z-z\nabla_{K_k}$
mod $z^{2k}$. Moreover, since $L_k$ is $C_{2k-1}$-equivalent to a
trivial link, $\nabla_{L_k}(z)\equiv 0$ mod $z^{2k-1}$. This implies
that $\nabla_{L_k}(z)\equiv -a_{2k-2}(K_k)z^{2k-1}$ mod $z^{2k}$.
Therefore, it is enough to compute $\nabla_{K_k}(z)$.

We compute the Alexander-Conway polynomial in order to have
$\nabla_{K_k}(z)$. For a Seifert surface $F$ of $K_k$ and a basis
$x_1,...,x_{2k-2},y_1,...,y_{2k-3},z$ of $H_1(F;{\Bbb Z})$ as
illustrated in Figure~\ref{example-fig4-1}, we have the following
Seifert matrix with respect to the basis 
\[M(K_k)=
	\left(
	\begin{array}{c||c|c}
		O_{(2k-2)\times(2k-2)} & 
		A_{(2k-2)\times(2k-3)} &
	\begin{array}{c}
		1\\	0\\	\vdots\\	0\\	-1
	\end{array}\\ 
	\hline\hline
		B_{(2k-3)\times(2k-2)} &
	\begin{array}{cccc}
		1&0&\cdots&0\\		0&0&\cdots&0\\
		\vdots&\vdots&&\vdots\\	0&0&\cdots&0
	\end{array} &
	\begin{array}{c}
		0\\	\vdots\\	\vdots\\	0\\
	\end{array}\\ 
	\hline
	\begin{array}{cccc}	0&\cdots&0&-1	\end{array} &
	\begin{array}{cccc}	0&0&\cdots&0	\end{array} &
		0
	\end{array}
	\right),\]
where $O_{(2k-2)\times(2k-2)}$ is the  $(2k-2)\times (2k-2)$ zero matrix,
$A_{(2k-2)\times(2k-3)}=(a_{ij})$ is a $(2k-2)\times (2k-3)$ matrix with
\[a_{ij}=
	\mathrm{lk}(x_i^+,y_j)=
	\left\{
	\begin{array}{ll}
	 	 1 & \text{if $i=j$,} \\
		-1 & \text{if $i\geq 3$ is odd and $j=i-1$,}\\
		 0 & \text{otherwise,}
	\end{array}
	\right. \]
and 
$B_{(2k-3)\times(2k-2)}=(b_{ij})$ is a $(2k-3)\times (2k-2)$ matrix with 
\[b_{ij}=
	\mathrm{lk}(y_i^+,x_j)=
	\left\{
	\begin{array}{ll}
		 1 & \text{if $i=j$,}\\
		-1 & \text{if $i$ is odd and $j=i+1$,}\\
		 0 & \text{otherwise.}
	\end{array}
	\right. \]

\begin{figure}[!h]
\includegraphics[trim=0mm 0mm 0mm 0mm, width=.8\linewidth]{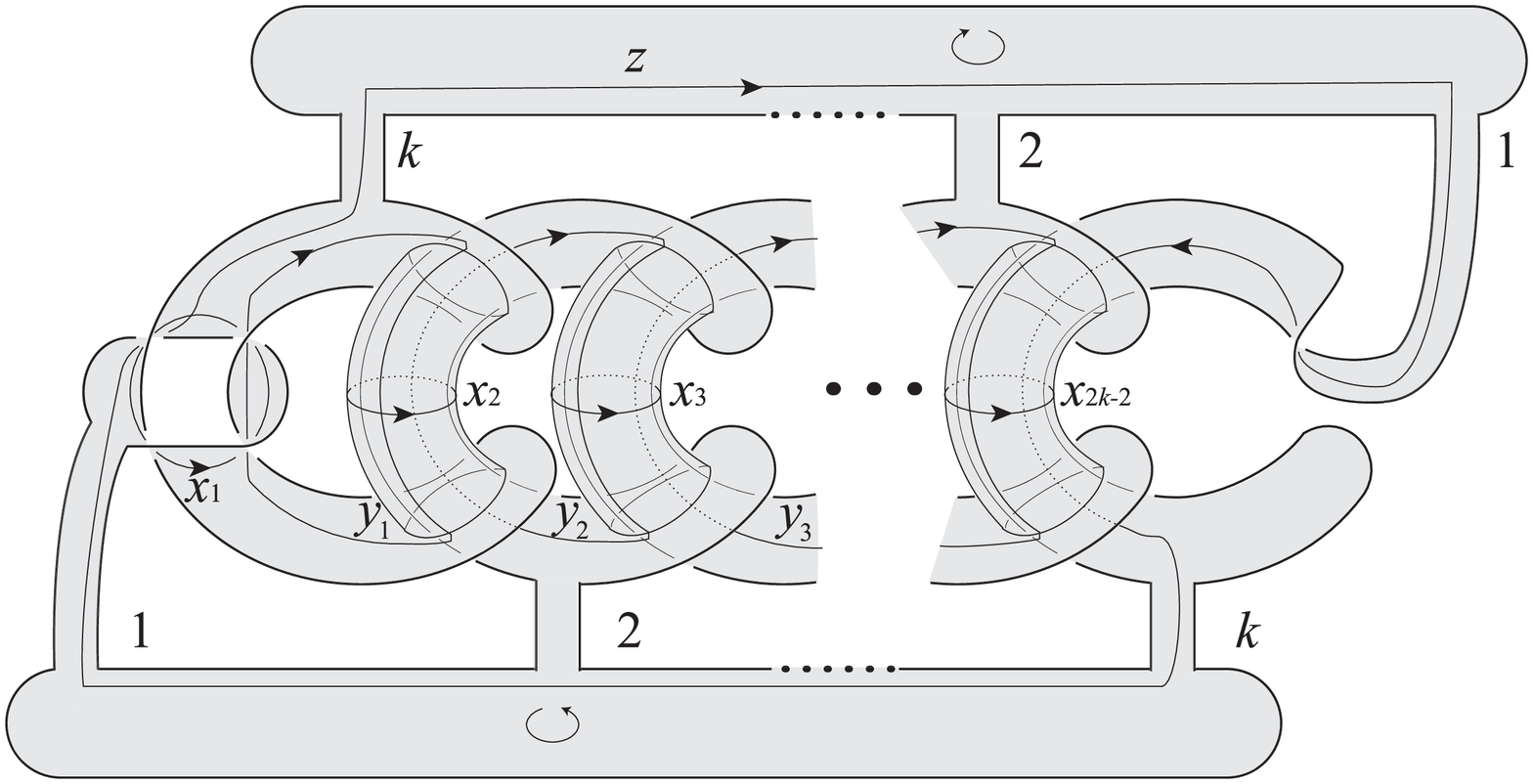}
\caption{}\label{example-fig4-1}
\end{figure}

For example, when $k=4$, then
\[A_{6\times 5}=
	\left(
	\begin{array} {ccccc}
		1     & 0    & 0 & 0 &0 \\
		0     & 1    & 0 & 0 &0 \\
		0     & -1   & 1 & 0 &0 \\
		0     & 0    & 0 & 1 &0 \\
		0     & 0    & 0 & -1&1 \\
		0     & 0    & 0 & 0 &0
	\end{array}
	\right),~\text{and}~
B_{5\times 6}=
	\left(
	\begin{array} {cccccc}
		1  & -1  & 0 & 0 & 0 & 0  \\
		0  &  1  & 0 & 0 & 0 & 0  \\
		0  & 0   & 1 & -1& 0 & 0  \\
		0  & 0   & 0 & 1 & 0 & 0  \\
		0  & 0  & 0 & 0 & 1 & -1
	\end{array}
	\right).\]
Then, 
the Conway polynomial $\nabla_{K_4}(\sqrt{t}^{-1}-\sqrt{t})=
|\sqrt{t}^{-1}M(K_4)-\sqrt{t}(M(K_4))^T|$ is the product of
	\[\left|
	\begin{array} {ccccc|c}
	\sqrt{t}^{-1}-\sqrt{t}     & 0    & 0 & 0 &0 &\sqrt{t}^{-1}\\
	\sqrt{t}    & \sqrt{t}^{-1}-\sqrt{t}    & 0 & 0 &0 &0\\
	0     & -\sqrt{t}^{-1}   & \sqrt{t}^{-1}-\sqrt{t} & 0 &0&0 \\
	0     & 0    & \sqrt{t} & \sqrt{t}^{-1}-\sqrt{t} &0&0 \\
	0     & 0    & 0 & -\sqrt{t}^{-1}&\sqrt{t}^{-1}-\sqrt{t} &0\\
	0     & 0    & 0 & 0 &\sqrt{t} &\sqrt{t}-\sqrt{t}^{-1}
	\end{array}
	\right|\]
and
	\[\left|
	\begin{array} {cccccc}
	\sqrt{t}^{-1}-\sqrt{t}  & -\sqrt{t}^{-1}  & 0 & 0 & 0 & 0  \\
	0  &  \sqrt{t}^{-1}-\sqrt{t}  & \sqrt{t} & 0 & 0 & 0  \\
	0  & 0   & \sqrt{t}^{-1}-\sqrt{t} & -\sqrt{t}^{-1}& 0 & 0  \\
	0  & 0   & 0 & \sqrt{t}^{-1}-\sqrt{t} & \sqrt{t} & 0  \\
	0  & 0  & 0 & 0 & \sqrt{t}^{-1}-\sqrt{t} & -\sqrt{t}^{-1}\\ \hline
	-\sqrt{t} & 0 & 0 & 0 & 0 & \sqrt{t}-\sqrt{t}^{-1}
	\end{array}
	\right|.
	\]
Hence we have
	\[\begin{array}{rcl}
	\nabla_{K_4}(\sqrt{t}^{-1}-\sqrt{t})
	&=&	
	((-1)^{3}-(\sqrt{t}^{-1}-\sqrt{t})^{6})
	((-1)^{5}-(\sqrt{t}^{-1}-\sqrt{t})^{6})\\
	&=&
	1+ 2(\sqrt{t}^{-1}-\sqrt{t})^{6}+(\sqrt{t}^{-1}-\sqrt{t})^{12}.
	\end{array}\]
In general,
	\[\begin{array}{rcl}
	\nabla_{K_k}(\sqrt{t}^{-1}-\sqrt{t})
	&=&
	((-1)^{k-1}-(\sqrt{t}^{-1}-\sqrt{t})^{2k-2})
	((-1)^{k+1}-(\sqrt{t}^{-1}-\sqrt{t})^{2k-2})\\
	&=&
	1+(-1)^k 2(\sqrt{t}^{-1}-\sqrt{t})^{2k-2}+
	(\sqrt{t}^{-1}-\sqrt{t})^{4k-4}.
	\end{array}\]
This implies
\[\nabla_{L_k}(z)\equiv - (-1)^k 2z^{2k-1}~\text{mod}~ z^{2k}.\]

On the other hand, we note that $L_k$ is obtained from the trivial
knot by surgery along $C_{2k-1}$-tree $T$ such that the number of
leaves that intersect the $i$th component is equal to $k$ for each
$i~(i=1,2)$ (see Figure~\ref{milnor-tangle}). It follows from the
proof of \cite[Lemma~1.2]{FY}  that each component of $L_k$ is
$C_{k-1}$-equivalent to the trivial knot in the complement of the
other component. Hence by Proposition~\ref{free-Ck-inv},
$\overline{\mu}_{L_k}(I)=0$ for any multi-index $I$ with entries
from $\{1,2\}$ such that either the index $1$ or $2$ appears in $I$
at most $k-1$ times. By \cite[Theorem~4.1]{Murasugi} (or
\cite[Theorem~4.1]{Cochran}),  we have 
\[(-1)^{k-1}\overline{\mu}_{L_k}([k,k])=
\sum_{p+q=2k}(-1)^{q-1}\overline{\mu}_{L_k}([p,q])=
-a_{2k-1}(L_k)=(-1)^k 2,\] and hence $\overline{\mu}_{L_k}([k,k])=-2$.
Proposition~\ref{free-Ck-inv} implies that each component of $L_k$
is not $C_{k}$-equivalent to the trivial knot in the complement of
the other component.
\end{proof}

We finish this section by presenting infinitely many pairs 
$L_{p}^{+} \cup L_{p}^{-}$ of componentwise satellite links 
of type $(\Gamma;1,p) (|p|\geq 2)$ such that $L_p^{+}$ is not
self-$\Delta$ concordant to $L_p^{-}$.

\begin{example} \label{example4}
Let $L_{p}^+$ (resp. $L_{p}^-$) be the link with linking number 
$p$ as illustrated in the left of Figure \ref{fig:ctr-expl2a} 
with $T_p^+$ (resp. $T_p^-$) representing the braid 
$\sigma_1\sigma_2\cdots\sigma_{|p|-1}$ (resp.
$\sigma_1^{-1}\sigma_2\cdots\sigma_{|p|-1}$) if $p>0$ and 
$\sigma_{|p|-1}\cdots\sigma_2\sigma_1$ (resp. 
$\sigma_{|p|-1}\cdots\sigma_2\sigma_1^{-1}$) if $p<0$. 
Note that both $L_p^+$ and $L_p^-$ are componentwise satellite 
links of type $(H;1,p)$ for the Hopf link $H$. $L_{p}^+$ and 
$L_{p}^-$ are not self-$\Delta$ concordant.
\end{example}

\begin{figure}[!h]
\includegraphics[trim=0mm 0mm 0mm 0mm, width=.8\linewidth]{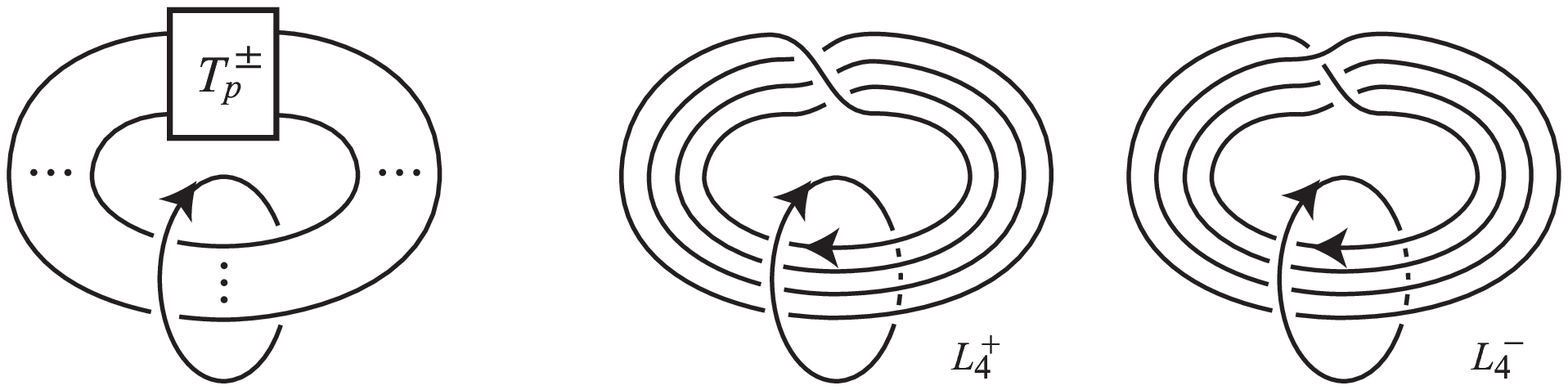}
\caption{}\label{fig:ctr-expl2a}
\end{figure}

\begin{proof}
Set $\varepsilon=p/|p|$. Let $L_{p}^0$ be the link obtained from 
$L_{p}^+$ by smoothing the crossing which corresponds to $\sigma_1$. 
Then by the definition of the Conway polynomial, we have
\[a_4(L_p^+)-a_4(L_p^-)=a_3(L_p^0),\] where $a_k$ is the coefficient 
of $z^k$ in the Conway polynomial. By \cite{JH-PAMS85}, we have
$a_3(L_p^0)=p-\varepsilon$. For a $2$-component link $L=K_1\cup K_2$,
it is known that $a_4(L)\equiv \overline{\mu}_L(1122)$ mod
$\overline{\mu}_L(12)$ \cite{Murasugi}, \cite{Cochran}, and
$\overline{\mu}_L(12)=\mathrm{lk}(K_1,K_2)=p$. Hence we have
\[\overline{\mu}_{L_p^+}(1122)-\overline{\mu}_{L_p^-}(1122)\equiv
a_4(L_p^+)-a_4(L_p^-)= a_3(L_p^0)=p-\varepsilon\equiv
-\varepsilon~~\text{mod}~p.\] Since $\overline{\mu}(1122)$ is a
self-$\Delta$ concordance invariant \cite{FY}, we have the conclusion.
\end{proof}

\begin{acknowledgments}
The authors would like to thank Professor Jonathan Hillman 
for pointing out Remark~\ref{remark3}.
\end{acknowledgments}

{\footnotesize
\bigskip
\noindent	Thomas FLEMING\\
		Department of Mathematics, University of California San Diego,
		9500 Gilman Dr., La Jolla, \\
		CA 92093-0112, United States\\
		e-mail: tfleming@math.ucsd.edu \par

\bigskip
\noindent	Tetsuo SHIBUYA \\
        	Department of Mathematics, Osaka Institute of Technology,
        	Asahi, Osaka 535-8585, Japan \\
        	e-mail: shibuya@ge.oit.ac.jp \par
\bigskip
\noindent	Tatsuya TSUKAMOTO \\
		Department of Mathematics, Osaka Institute of Technology,
		Asahi, Osaka 535-8585, Japan \\
        	e-mail: tsukamoto@ge.oit.ac.jp \par
\bigskip
\noindent	Akira YASUHARA \\
        	Department of Mathematics, Tokyo Gakugei University,
        	Koganei, Tokyo 184-8501, Japan \\
        	e-mail: yasuhara@u-gakugei.ac.jp
}

\end{document}